\documentclass[reqno,12pt]{amsart}

\parskip = \bigskipamount
\usepackage{amssymb}
\usepackage{amscd}
\usepackage[all]{xy}
\usepackage{bbm}
\usepackage{mathrsfs}
\usepackage{enumerate}
\usepackage{tikz}
\usepackage[margin=1.5in]{geometry}
\usepackage{stmaryrd}
\usepackage{etoolbox}
\usepackage{array}

\newcommand{\R}{\mathbb{R}}

\newcommand{\inv}{^{-1}}

\newcommand{\eps}{\varepsilon}

\newcommand{\del}{\nabla}

\newcommand{\lap}{\Delta}

\newcommand{\bd}{\partial}
\newcommand{\cl}{\overline}

\newcommand{\eval}{\bigg\vert}

\newcommand{\la}{\langle}
\newcommand{\ra}{\rangle}

\renewcommand{\div}{\operatorname{div}}

\newcommand{\grad}{\del}
\newcommand{\vol}{\operatorname{Vol}}

\newcommand{\f}{\colon}
\newcommand{\area}{\operatorname{Area}}

\theoremstyle{plain}
\newtheorem{theorem}{Theorem}
\newtheorem{corollary}[theorem]{Corollary}
\newtheorem{prop}[theorem]{Proposition}
\newtheorem{lem}[theorem]{Lemma}

\theoremstyle{definition}
\newtheorem{defn}[theorem]{Definition}
\newtheorem{rem}[theorem]{Remark}

\begin{document}

\title[PMCs in Non-compact manifolds]{Prescribed Mean Curvature Min-Max Theory in Some Non-compact Manifolds}
\author{Liam Mazurowski}
\address{Cornell University, Department of Mathematics, Ithaca, New York 14850}
\email{lmm334@cornell.edu}

\begin{abstract}
This paper develops a technique for applying one-parameter prescribed mean curvature min-max theory in certain non-compact manifolds.  We give two main applications.  First, fix a dimension $3 \le n+1 \le 7$ and consider a smooth function $h\f \R^{n+1}\to \R$ which is asymptotic to a positive constant near infinity.  We show that, under certain additional assumptions on $h$, there exists a closed hypersurface $\Sigma$ in $\R^{n+1}$ with mean curvature prescribed by $h$.  Second, let $(M^3,g)$ be an asymptotically flat 3-manifold and fix a constant $c > 0$.  We show that, under an additional assumption on $M$, it is possible to find a closed surface $\Sigma$ of constant mean curvature $c$ in $M$. 
\end{abstract}

\maketitle

\section{Introduction}

\subsection{PMC Surfaces in Euclidean Space} In his list of problems in differential geometry \cite{Y}, Yau proposed the following question: 
\begin{itemize}
\item[] {For which functions $h\f \R^3\to \R$ does there exist an embedded 2-sphere with mean curvature prescribed by $h$?}
\end{itemize}
In this paper we consider a closely related question.  Namely, given a dimension $3 \le n+1 \le 7$, for which functions $h\f \R^{n+1}\to \R$ does there exist a closed, embedded hypersurface of mean curvature prescribed by $h$? Throughout the rest of the paper we will refer to such a surface as an $h$-PMC\footnote{We do not require an $h$-PMC to be topologically a sphere.  This is a natural assumption from the point of view of geometric measure theory.}.

In dimension three, Yau's question can be investigated using the mapping method. Indeed,  branched immersed $h$-PMCs can be constructed by looking for maps $u\f S^2\to \R^3$ which are critical for the Dirichlet energy minus an $h$-weighted volume term.  
 Caldiroli and Musina have used this approach to give sufficient conditions for the existence of an $h$-PMC. They consider a function $h\f \R^3\to \R$ which is asymptotic to a constant at infinity and whose radial derivative satisfies a certain decay condition.  In \cite{CalMu},  using a one parameter min-max argument, they prove that an $h$-PMC exists provided an appropriate mountain pass energy level is not attained by spheres at infinity.  In a later paper \cite{CalMu1}, they show that even when the mountain pass energy level is attained by spheres at infinity, it is still sometimes possible to obtain an $h$-PMC via a higher parameter min-max argument.

The mapping method does not work in higher dimensions.  However, the techniques of geometric measure theory often provide a viable alternative.  Min-max schemes based on geometric measure theory were developed for the area functional in the early 1980s by Almgren \cite{A}, Pitts \cite{P}, and Schoen-Simon \cite{SS}.
They have since been used with great success in the construction of minimal surfaces by Marques, Neves, Song and others (see for example \cite{IMN}, \cite{Li}, \cite{LMN}, \cite{MNwil}, \cite{MN}, \cite{MNric}, \cite{MN2}, \cite{MNS}, \cite{Song}, \cite{SZ}).  Min-max schemes capable of finding prescribed mean curvature surfaces were developed more recently by Zhou and Zhu (see \cite{ZZ}  and \cite{ZZ1}).  Their prescribed mean curvature min-max theory has already had a number of applications.  For example, Zhou \cite{Z} used it to regularize the area functional, and thus obtained a proof of the Multiplicity One Conjecture of Marques and Neves.  (Chodosh and Mantoulidis \cite{CM} had previously proven the Multiplicity One Conjecture in ambient dimension 3 using the Allen-Cahn theory.) Also Dey \cite{Dey} used the PMC min-max theory to produce many surfaces of constant mean curvature $c$ in a given manifold $M$ when $c$ is small, and the author used it to give a new construction of constant mean curvature doublings of minimal surfaces \cite{Maz1}.  For a more complete overview of the PMC min-max theory, see Zhou's survey \cite{ZhouSurvey}. 

In Section \ref{pmc} of this paper, we use Zhou and Zhu's min-max theory to investigate the existence of $h$-PMCs on $\R^{n+1}$ for $3 \le n+1\le7$.  Our first main result can be viewed as an analog of \cite{CalMu} in higher dimensions. 

\begin{theorem}
\label{main}
Fix a dimension $3 \le n+1\le 7$ and let $h\f \R^{n+1}\to \R$ be a smooth function satisfying hypotheses $(\operatorname{H1})$-$(\operatorname{H4})$.  Then there exists a smooth, almost embedded $h$-PMC in $\R^{n+1}$. 
\end{theorem}

Here a surface $\Sigma$ is called almost embedded provided in a neighborhood of each point $p\in \Sigma$ either 
\begin{itemize}
\item[(i)] $\Sigma$ is embedded, or
\item[(ii)] $\Sigma$ decomposes into an ordered union of embedded sheets. 
\end{itemize} 
Roughly speaking, the hypotheses (H1)-(H4) are as follows:
\begin{itemize}
\item[(H1)] The function $h$ is asymptotic to a positive constant $c$ at infinity. 
\item[(H2)] The radial derivative of $h$ is smaller than $c/\| x\|$ near infinity.
\item[(H3)] The function $h$ is admissible for PMC min-max theory. 
\item[(H4)] The mountain pass energy level for $h$ is not attained by  spheres at infinity.
\end{itemize}  
See Section \ref{pmc} for the precise formulation of these hypotheses.

%\begin{rem}  
It is worth comparing Theorem \ref{main} with the result of \cite{CalMu} in the case that $n+1 = 3$.  First, our assumption (H2)  is weaker than the corresponding assumption in \cite{CalMu} which asks that the radial derivative be bounded by $c/\|x\|^2$ everywhere.  Also the $h$-PMCs we obtain are almost embedded, whereas the $h$-PMCs in \cite{CalMu} are branched immersed. On the other hand, the $h$-PMCs constructed in \cite{CalMu} have the topology of $S^2$ whereas there is no control over the topology of the surfaces produced by Theorem \ref{main}.
%\end{rem} 

\subsection{CMC Surfaces in Asymptotically Flat Manifolds} The proof of Theorem 1 uses a cutoff trick to reduce min-max on $\R^{n+1}$ to min-max on a sequence of larger and larger balls. We are hopeful that this cutoff trick may be adapted to perform min-max in other non-compact manifolds.  As an example, in Section \ref{cmc}, we show that essentially the same argument can be used to produce constant mean curvature surfaces in some asymptotically flat manifolds.  This partially answers a question posed by Zhou in \cite{ZhouSurvey}. 

Recall that a manifold $(M^3,g)$ is called asymptotically flat if there is a compact set $K$ such that $M^3\setminus K$ is diffeomorphic to $\R^3\setminus B_1(0)$, and moreover, the metric $g$ is asymptotic to the Euclidean metric in the coordinates induced by this diffeomorphism.  Asymptotically flat manifolds are important in general relativity, where they model the gravitational field of an isolated object.
Huisken and Yau \cite{HY} showed that manifolds asymptotic to a Schwarzchild solution of positive mass admit a foliation by constant mean curvature (CMC) surfaces near infinity.  See Ye \cite{YeAF} for another proof.  This existence result has since been generalized by Eichmair-Koerber \cite{EK}, Huang \cite{Huang}, Metzger \cite{Metzger}, Nerz \cite{Nerz}, and others to allow for weaker decay of the metric to the flat metric near infinity.  The uniqueness of such CMC foliations has also been an active topic of research.  See Eichmair-Koerber \cite{EK} for a comprehensive list of references.  

Note that the spheres in these foliations near infinity are very large and hence have small mean curvature.  Ye \cite{Ye} showed that manifolds admit foliations by small CMC spheres near non-degenerate critical points of the scalar curvature.  Of course, these spheres have very large mean curvature.  Thus, in general, an asymptotically flat manifold contains closed CMCs for very small values of the mean curvature as well as very large values of the mean curvature.  It is natural to ask whether such a manifold also contains CMCs for intermediate values of the mean curvature. Our second main result shows that, under an additional assumption on $M$, this is indeed the case. 

\begin{theorem}
\label{main2}
Fix a constant $c > 0$.  Let $(M^3,g)$ be an asymptotically flat manifold which satisfies hypothesis $(H)$ for the constant $c$. Then there exists a closed, almost-embedded $c$-CMC in $M$. 
\end{theorem}

Here hypothesis (H) asks that a certain mountain pass energy level is not attained by spheres at infinity.  See Section \ref{cmc} for the precise formulation of hypothesis (H), as well as the exact definition of an asymptotically flat manifold. 

\subsection{Outline of Proof} In the remainder of the introduction we give a heuristic outline of the proof of Theorem \ref{main}.  Let $B_r$ denote the ball of radius $r$ centered at the origin in $\R^{n+1}$.  Let $h\f \R^{n+1}\to \R$ be a smooth function which satisfies (H1)-(H3).    Without loss of generality, we can assume that $h$ is asymptotic to $n$ at infinity.    We make this choice for convenience because the mean curvature of the unit sphere in $\R^{n+1}$ is equal to $n$. 

Given an open set $\Omega$ in $\R^{n+1}$ with smooth boundary, define 
\[
A^h(\Omega) = \area(\bd \Omega) - \int_{\Omega} h.
\]
A set $\Omega$ with smooth boundary is a critical point for $A^h$ precisely when $\bd \Omega$ is an $h$-PMC.  Call a smoothly varying family of open sets $\{\Omega_t\}_{t\in[0,1]}$ a mountain pass path if $\Omega_0 = \emptyset$ and $A^h(\Omega_1) < 0$.  Define the mountain pass energy level
\[
\omega = \inf_{\{\Omega_t\}} \left[\sup_{t\in [0,1]} A^h(\Omega_t)\right]
\]
where the infimum is taken over all mountain pass paths $\{\Omega_t\}_{t\in[0,1]}$.  
The strategy is to try and obtain an $h$-PMC $\Sigma = \bd \Omega$ with $A^h(\Omega) = \omega$ by using Zhou and Zhu's prescribed mean curvature min-max theory. 

The major difficulty with this approach is that there are no barriers at infinity for the $A^h$ functional.  Indeed, without further assumptions on $h$, it may happen that min-max detects a sphere at infinity.     For example, suppose that $h$ is radially increasing.  Then $\Omega_t^x = B_{rt}(x)$ is a mountain pass path provided $\|x\|$ and $r$ are large enough.  Moreover
\[
\sup_{t\in [0,1]} A^h(\Omega_t^x) \to \omega = A^n(B_1)
\]
as $\|x\|\to \infty$.   In this example, every min-max sequence converges to a sphere of radius 1 at infinity.  

This particular behavior can be ruled out by requiring that $\omega < A^n(B_1)$.  This is hypothesis (H4).  However, even with hypothesis (H4), the lack of barriers makes it difficult to directly perform min-max.  
To get around this, we introduce a family of approximate problems which do admit barriers.  Namely, for each $R > 0$, let $h_R$ be a function which agrees with $h$ on $B_R$, falls off linearly on $B_{R+1}\setminus B_R$, and is identically 0 outside $B_{R+1}$. Then large spheres centered at the origin serve as  barriers for $h_R$-PMCs.  Consequently, it is possible to apply the above min-max scheme with $A^{h_R}$ in place of $A^h$ to produce an ${h_R}$-PMC $\Sigma_R = \bd \Omega_R$ contained in $B_{R+1}$.  Using hypothesis (H2) together with Str\"uwe's monotonicity trick, we verify that the surfaces $\Sigma_R$ have uniform area bounds.  Moreover, the surfaces $\Sigma_R$ have index at most one for $A^{h_R}$.  

To complete the proof, it suffices to show that $\Sigma_R$ does not intersect the transition region $B_{R+1}\setminus B_R$ when $R$ is large.  Suppose to the contrary that every surface $\Sigma_R$ intersects the transition region $B_{R+1}\setminus B_R$.   Translate each $\Sigma_R$ to get a new surface $\Sigma_R'$ that passes through the origin.   By the area and index bounds, these surfaces $\Sigma_R'$ converge to a limiting surface $\Sigma_\infty = \bd \Omega_\infty$ as $R\to \infty$.  The limiting surface $\Sigma_\infty$ is an $h_\infty$-PMC where the function $h_\infty$ is $n$ on one side of a slab, falls off linearly across the slab, and is 0 on the other side of the slab.   The montonicity of $h_\infty$ implies that $\Sigma_\infty$ must be a sphere of radius 1 contained in the region where $h_\infty = n$. This is a contradiction to hypothesis (H4) because $A^n(B_1) = A^{h_\infty}(\Omega_\infty) = \lim_{R\to \infty} A^{h_R}(\Omega_R) = \omega$. 

\subsection{Organization} The remainder of the paper is organized as follows.  Section 2 reviews the necessary background from geometric measure theory and sets the notation for the remainder of the paper.   Section 3 proves Theorem \ref{main} on the existence of $h$-PMCs in Euclidean space.  Finally, Section 4 proves Theorem \ref{main2} on the existence of CMCs in asymptotically flat manifolds.

\subsection{Acknowledgements} The author would like to thank Xin Zhou for introducing him to the problem and for many valuable discussions.  He would also like to thank Andr\'e Neves for his continued encouragement.

\section{Background and Notation}

We will need the following concepts from geometric measure theory.   Let $M$ be a complete Riemannian manifold.  In this paper, $M$ will always be either Euclidean $\R^{n+1}$ or an asymptotically flat manifold $(M^3,g)$.  Let $\mathcal C(M)$ denote the set of all Caccioppoli sets with compact support in $M$.  The flat topology on $\mathcal C(M)$ is induced by the metric 
$
\mathcal F(\Omega_1,\Omega_2) = \vol(\Omega_1 \Delta \Omega_2)
$
where $\Delta$ denotes the symmetric difference.  For $\Omega\in \mathcal C(M)$ the notation $\bd \Omega$ stands for the boundary of $\Omega$ in the sense of flat chains mod two.  The $\mathbf F$-topology on $\mathcal C(M)$ is given by 
\[
\mathbf F(\Omega_1,\Omega_2) = \mathcal F(\Omega_1,\Omega_2) + \mathbf F(\vert \bd \Omega_1\vert, \vert \bd \Omega_2\vert). 
\]
Here $\vert \bd \Omega_i\vert$ denotes the varifold induced by the flat chain $\bd \Omega_i$ and $\mathbf F$ is the $\mathbf F$-metric on varifolds (see \cite{P}).

Let $h \f M\to \R$ be a smooth function.

\begin{defn}
The functional $A^h\f \mathcal C(M)\to \R$ is given by 
\[
A^h(\Omega) = \area(\bd \Omega) - \int_{\Omega} h.
\] 
\end{defn}

\begin{prop}
The first variation of $A^h$ at a set $\Omega$ with smooth boundary is given by 
\[
\delta A^h\eval_\Omega(X) = -\int_{\bd \Omega} H\cdot X + \int_{\bd \Omega} h X\cdot \nu.
\]
Here $H$ denotes the mean curvature vector for $\bd \Omega$ and $\nu$ is the inward normal to $\bd \Omega$. 
\end{prop}

It is clear from the previous formula that a smooth critical point of $A^h$ gives rise to an $h$-PMC.  In the special case where $h = c$ is constant, the  functional simplifies to 
$
A^c(\Omega) = \area(\bd \Omega) - c\vol(\Omega).
$
Smooth critical points of the $A^c$ functional give rise to $c$-CMCs.

Next we give the basic definitions regarding the min-max scheme for $A^h$. 

\begin{defn}
Let $\{\Omega_t\}_{t\in [0,1]}$ be a family in $\mathcal C(M)$ varying continuously in the $\mathbf F$-topology.  The family $\{\Omega_t\}$ is called an $h$-mountain pass path provided $\Omega_0 = \emptyset$ and $A^h(\Omega_1) < 0$.  Let $\mathcal P$ denote the set of all $h$-mountain pass paths.
\end{defn}

\begin{defn}
The one parameter min-max width for $A^h$ is the number 
\[
\omega = \inf_{\{\Omega_t\}\in \mathcal P} \left[\sup_{t\in [0,1]} A^h(\Omega_t)\right]. 
\]
\end{defn}

\section{PMCs in $\R^n$} 
\label{pmc} 

\subsection{Set Up} Fix a dimension $3\le n+1 \le 7$ and let $M$ be $\R^{n+1}$ equipped with the Euclidean metric.   Let $h\f \R^{n+1}\to \R$ be a smooth function which satisfies the following hypotheses:
\begin{itemize}
\item[(H1)] The function $h$ converges smoothly to a constant $c >0$ at infinity;
\item[(H2)] There are numbers $\rho > 0$ and $0 < \sigma < 1$ such that $\vert \grad h(x)\cdot x\vert < \sigma h(x)$ whenever $\|x\| \ge \rho$.  
\item[(H3)] The function $h$ satisfies the admissibility condition $(\dagger)$ from \cite{ZZ1}.
\end{itemize}

Let $\mathcal P$ be the set of all $h$-mountain pass paths.  It is clear that $\mathcal P$ is non-empty.  Indeed, the path $t\mapsto B_{tr}(x)$ is a mountain pass path provided $\|x\|$ and $r$ are large enough.  Let $\omega$ be the min-max width for $A^h$ as in the previous section.  Then $\omega$ satisfies the following basic estimate. 

\begin{prop}
\label{estmm}
The one parameter min-max width satisfies $0 < \omega \le A^c(B_{n/c})$. 
\end{prop}

\begin{proof}
The fact that $\omega > 0$ is a consequence of the isoperimetric inequality.  In more detail, the isoperimetric inequality says there is a constant $C$ depending only on the dimension such that 
$
\area(\bd \Omega) \ge C \vol(\Omega)^{n/({n+1})}
$
for all $\Omega\in \mathcal C(\R^{n+1})$.  Let $M = \sup_{\R^{n+1}} h < \infty$.  Then
\[
A^h(\Omega) \ge C \vol(\Omega)^{\frac n {n+1}} - \int_\Omega h \ge \vol(\Omega)^{\frac n{n+1}} \left[C - M \vol(\Omega)^{\frac 1{n+1}}\right]
\]
for all $\Omega\in \mathcal C(\R^{n+1})$.  In particular, there are positive constants $v,a>0$ such that $A^h(\Omega) \ge 0$ whenever $\vol(\Omega) \le v$ and $A^h(\Omega) \ge a$ whenever $\vol(\Omega) = v$. It follows that $\omega \ge a > 0$. 

The upper bound on $\omega$ is proved by analyzing explicit mountain pass paths near infinity.  Fix a point $x\in \R^{n+1}$ with $\|x\|$ large.  Also fix a very large number $r$.  Then $\Omega_t = B_{rt}(x)$ is a mountain pass path.  Moreover, 
\begin{align*}
A^h(\Omega_t) &= \area(\bd \Omega_t) -\int_{\Omega_t} h\\
&= \area(\bd B_{rt}) - c \vol(B_{rt}) - \int_{\Omega_t} (h-c).
\end{align*} 
Elementary calculus shows that $\area(\bd B_{rt}) - c\vol(B_{rt})$ is maximized for $t = \frac{n}{cr}$ and the maximum value is $A^c(B_{n/c})$.   Thus 
\begin{align*}
A^h(\Omega_t) &\le A^c(B_{n/c}) + \vol(\Omega_t) \sup_{\Omega_t} \vert h-c\vert \\
&\le A^c(B_{n/c}) + \vol(B_{r}) \sup_{\R^{n+1} \setminus B} \vert h - c\vert,
\end{align*}
where $B$ is a ball of radius $\|x\|-r$ centered at the origin.
Since $t$ was arbitrary and 
\[
\sup_{\R^{n+1} \setminus B} \vert h-c\vert \to 0
\]
as $\|x\| \to \infty$ by (H1), it follows that $\omega \le A^c(B_{n/c})$. 
\end{proof}

Throughout the rest of the paper we make the following additional assumption on $h$.  
\begin{itemize}
\item[(H4)] The width $\omega$ is strictly less than $A^c(B_{n/c})$.
\end{itemize}
This will be the case, for example, if $h(x) > c$  for all $x$ in some sufficiently large ball.

\begin{rem}
For the remainder of the paper we will assume that the constant $c$ is equal to $n$. It is easy to see that this causes no loss of generality.
For the reader's convenience we now summarize the hypotheses (H1)-(H4) in the case that $c = n$. 
\begin{itemize}
\item[(H1)] The function $h$ converges smoothly to the constant $n$ at infinity. 
\item[(H2)] There are numbers $\rho > 0$ and $0<\sigma<1$ such that $\vert \grad h(x)\cdot x\vert < \sigma h(x)$ whenever $\|x\|\ge \rho$. 
\item[(H3)] The function $h$ satisfies the admissibility condition $(\dagger)$ from \cite{ZZ1}.
\item[(H4)] The width $\omega$ is strictly less than $A^n(B_1)$.
\end{itemize}
\end{rem}

\subsection{The Approximate Problems} 

Fix a function $h$ satisfying (H1)-(H4).  In this section, we define and solve a sequence of approximate problems for which it is easy to confine the corresponding min-max surface inside a compact region of $\R^{n+1}$.  

\begin{defn} 
\label{zetadef} 
Let $\zeta \f \R\to \R$ be a smooth, decreasing function such that 
\[
\zeta(r) = \begin{cases}
1, &\text{if } r \le 0\\
0, &\text{if } r\ge 1
\end{cases}
\]
For each positive number $R$ define $\zeta_R \f \R^{n+1}\to \R$ by $\zeta_R(x) = \zeta(\|x\|-R)$.  Finally let $h_R = \zeta_R h$. 
\end{defn}

\begin{defn}
Let $\mathcal P_R$ denote the set of all $h_R$-mountain pass paths, and let $\omega(R)$ be the one parameter min-max width for $h_R$. 
\end{defn}

Next we record several important properties of the width $\omega(R)$.  The following lemma is a consequence of the the fact that large spheres centered at the origin serve as a barrier to $h_R$-PMCs.

\begin{defn}
\label{epsr} 
For each $R > 0$, select a number $\eps_R > 0$ so that 
\[
\zeta(r) < \frac{1}{4(r+ R)} \quad \text{for all } r \ge 1 - 2 \eps_R. 
\]
Without loss of generality, $\eps_R$ can be chosen to be a decreasing function of $R$.  For $R$ sufficiently large, this choice of $\eps_R$ guarantees that
$
h_R(x) < {n}/{\|x\|} 
$
for $\|x\| \ge R+1-2\eps_R$. 
\end{defn}

\begin{lem}
\label{intersect}
Let $\Omega \in \mathcal C(\R^{n+1})$.  Then $\Omega \cap B_{R+1-2\eps_R}$ also belongs to $\mathcal C(\R^{n+1})$ and satisfies $A^{h_R}(\Omega\cap B_{R+1-2\eps_R}) \le A^{h_R}(\Omega)$. 
\end{lem}

\begin{proof}
Let $B = B_{R+1-2\eps_R}$.  The intersection of two Caccioppoli sets is a Caccioppoli set.  Therefore 
$
\Omega \cap B
$
belongs to $\mathcal C(\R^{n+1})$. The vector field $X(x) = x/\|x\|$ satisfies 
$
\div(X) = {n}/{\|x\|}.
$
Thus $h_R(x) < \div(X)$ on the complement of $B$.  Hence the divergence theorem implies that $A^{h_R}(\Omega \cap B) \le A^{h_R}(\Omega)$.
\end{proof}

\begin{prop}
\label{support} 
Let $\mathcal Q_R$ be the collection of paths $\{\Omega_t\}\in \mathcal P_R$ such that $\Omega_t$ is supported in $B_{R+1-\eps_R}$ for all $t\in [0,1]$. Then 
\[
\omega(R) = \inf_{\{\Omega_t\} \in \mathcal Q_R} \left[\sup_{t\in[0,1]} A^{h_R}(\Omega_t)\right].
\] 
In other words, the min-max value does not change if we restrict to only those mountain pass paths which are supported in $B_{R+1-\eps_R}$. 
\end{prop} 

\begin{proof} 
Fix a number $\eps > 0$.  Choose a mountain pass path $\{\Omega_t\}\in \mathcal P_R$ such that 
\[
\sup_{t\in [0,1]} A^{h_R}(\Omega_t) \le \omega(R) + \eps.
\]
Define $U_t = \Omega_t \cap B_{R+1-2\eps_R}$.  By Lemma \ref{intersect}, each $U_t$ belongs to $\mathcal C(\R^{n+1})$ and satisfies $A^{h_R}(U_t) \le A^{h_R}(\Omega_t)$.  Moreover, as shown in \cite{Dey},  $U_t$ varies continuously in the flat topology and has no concentration of mass (as defined in \cite{MNric}).  Hence, by the interpolation theorems of Zhou \cite{Z}, there is an $\mathbf F$-continuous family $\{W_t\}_{t\in[0,1]}$ in $\mathcal C(\R^{n+1})$ such that
\begin{itemize}
\item[(i)] $W_0 = \emptyset$ and $W_1 = U_1$, 
\item[(ii)] $W_t$ is supported in $B_{R+1-\eps_R}$ for all $t$, and  \item[(iii)] $A^{h_R}(W_t) \le A^{h_R}(U_t)+\eps$ for all $t$. 
\end{itemize}
This family $\{W_t\}$ belongs to $\mathcal Q_R$ and satisfies 
\[
\sup_{t\in[0,1]} A^{h_R}(W_t) \le \omega(R) + 2\eps.
\]
Since $\eps$ was arbitrary, the proposition follows.
\end{proof}

The next propositions establish the monotonicity of $\omega(R)$ and show that $\omega(R) \to \omega$ as $R\to \infty$. 

\begin{prop}
The width $\omega(R)$ is a decreasing function of $R$. 
\end{prop}

\begin{proof}
Fix two numbers $R < R'$.  Also fix some small $\eps >0$.  By definition, there is a family $\{\Omega_t\} \in \mathcal P_{R}$ such that 
\[
\sup_{t\in [0,1]} A^{h_R}(\Omega_t) \le \omega(R) + \eps.
\]
Since $\{\Omega_t\} \in \mathcal P_{{R'}}$ and $A^{h_R}(\Omega_t) \le A^{h_{R'}}(\Omega_t)$ for all $t\in [0,1]$, this implies that 
\[
\omega({R'}) \le \sup_{t\in [0,1]} A^{h_R'}(\Omega_t) \le \omega(R) + \eps.
\]
The result follows since $\eps$ was arbitrary.
\end{proof}

\begin{prop}
\label{radii}
The widths $\omega(R)$ satisfy $\omega(R) \to \omega$ as $R\to \infty$. In particular, $\omega(R) < A^n(B_1)$ when $R$ is large by $(\operatorname{H4})$. 
\end{prop} 

\begin{proof}
It is clear that $\omega(R) \ge \omega$ for all $R$.  Since $\omega(R)$ is decreasing, it is enough to show that for every $\eps >0$ there is an $R$ such that $\omega(R) \le \omega+\eps$.  So fix a small number $\eps > 0$. Choose a mountain pass path $\{\Omega_t\} \in \mathcal P$ for which 
\[
\sup_{t\in [0,1]} A^h(\Omega_t) \le \omega + \eps. 
\] 
We claim that 
\begin{equation}
\label{1} 
\sup_{t\in [0,1]} \vol(\Omega_t \setminus B_R) < \eps 
\end{equation}
for $R$ large enough. Indeed, this follows from the fact that $\vol(\Omega_t\setminus B_R)$ is a continuous function of $(t,R)\in [0,1]\times(0,\infty)$, the fact that $[0,1]$ is compact, and the fact that $\vol(\Omega_t \setminus B_R)\to 0$ as $R\to \infty$ for each fixed $t\in [0,1]$.

Choose $R$ large enough that (\ref{1}) holds.  By making $R$ even larger if necessary, we can assume that $\Omega_1\subset B_R$. Then $\{\Omega_t\}\in \mathcal P_{R}$ and 
\begin{align*}
\sup_{t\in [0,1]} A^{h_R} (\Omega_t) &= \sup_{t\in [0,1]} \left[\area(\bd \Omega_t) - \int_{\Omega_t} h_R \right] \\ &\le \omega + \eps + 2n \vol(\Omega_t \setminus B_R) \le \omega + (1+2n)\eps. \phantom{\int}
\end{align*}
Thus $\omega(R) \le \omega + (1+2n)\eps$ and the proposition follows.
\end{proof}

In order to establish area bounds on the min-max $h_R$-PMCs, it is necessary to construct mountain pass paths which have somewhat controlled behavior in the transition region $B_{R+1}\setminus B_R$. 

\begin{defn}
For each $R > 0$, $C > 0$, $\eta > 0$, and $\theta>0$, define $\mathcal N_{R,C,\eta,\theta}$ to be the set of paths $\{\Omega_t\}_{t\in[0,1]}\in \mathcal P_R$ such that the following properties hold:
\begin{itemize}
\item[(i)] every $\Omega_t$ is supported in $B_{R+1-\eps_R}$,
\item[(ii)] $\sup_{t\in[0,1]} A^{h_R}(\Omega_t) \le \omega(R) + \theta$,
\item[(iii)] $\int_{\Omega_t} \vert \grad \zeta_R \vert h \le C$
for all $t$ such that $A^{h_R}(\Omega_t) \ge \omega(R)-\eta$.
\end{itemize}
\end{defn}

The next propositions exploit the monotonicity of $\omega(R)$ in order to construct mountain pass paths that belong to the nice class $\mathcal N_{R,C,\eta,\theta}$.  This technique is inspired by Str\"uwe's monotonicity trick \cite{Struwe}.  The monotonicity trick was also used by Cheng and Zhou in their construction of constant mean curvature 2-spheres inside of Riemannian 3-spheres \cite{ChengZhou} as well as their construction of curves with constant geodesic curvature in Riemannian 2-spheres \cite{ChengZhou1}. 

\begin{prop}
\label{dbound} 
There is a sequence $R_j\to \infty$ such that $
0 \ge \omega'(R_j) \ge -{2}/{R_j}$
for all $j$.  
\end{prop}

\begin{proof}
The function $\omega(R)$ is a decreasing function of $R$ and consequently it is differentiable almost everywhere.  Moreover, for any real numbers $0<a<b$ there is an estimate
\[
\int_a^b \omega'(R) \, dR \ge \omega(b) - \omega(a) \ge \omega - \omega(a).
\]
Select $a$ large enough that $\omega(a) - \omega \le 1$.  Then for any positive integer $n\ge a$ we have 
\[
0 \ge \int_n^{2n} \omega'(R)\, dR \ge -1
\]
and consequently there must be an $R \in [n,2n]$ such that $\vert \omega'(R)\vert \le 2/R$. This implies the result.
\end{proof}

\begin{prop}
\label{nices} 
Assume that  $0 \ge \omega'(R) \ge -c$ for some constant $c > 0$.  Then there is an $\eta>0$ such that the set $\mathcal N_{R,5c,\eta,\theta}$ is non-empty for all sufficiently small $\theta$.
\end{prop}

\begin{proof}
For each positive integer $k$, let $R_k = R -k\inv$. Then 
\[
0 \ge \lim_{k\to \infty} \left(\frac{\omega(R_k)-\omega(R)}{R_k-R}\right) \ge - c
\]
and so 
\begin{equation}
\label{eq7}
0 \le \frac{\omega(R_k) - \omega(R)}{k\inv} \le {2c} 
\end{equation}
for all sufficiently large $k$.  By Proposition \ref{support}, for every $\theta > 0$, it is possible to choose a path $\{\Omega_t\}_{t\in[0,1]} \in  {\mathcal Q_{R_k}}$ such that 
\[
\sup_{t\in [0,1]} A^{h_{R_k}}(\Omega_t) \le \omega(R_k) + \theta. 
\]
Combined with (\ref{eq7}), this yields 
\begin{equation}
\label{eq2}
\sup_{t\in [0,1]} A^{h_{R_k}}(\Omega_t) \le \omega(R) + 2c k\inv + \theta.
\end{equation}
Now notice that  
\[
A^{h_{R_k}}(\Omega_t)   =  A^{h_R}(\Omega_t) + \int_{\Omega_t} h_{R} - h_{R_k}
\]
and so 
\begin{equation}
\label{eq3}
\sup_{t\in [0,1]}\left[ A^{h_R}(\Omega_t) + \int_{\Omega_t} h_{R} - h_{R_k}
\right] \le \omega(R) + 2c k\inv + \theta.
\end{equation}
Also observe that $\{\Omega_t\}_{t\in [0,1]}$ belongs to $\mathcal P_R$.  

We claim $\{\Omega_t\}_{t\in[0,1]}$ actually belongs to the class $\mathcal N_{R,5c,c k\inv, \theta}$ provided $k$ is large enough and $\theta$ is small enough (depending on $k$).  To see this, suppose that $\Omega_t$ satisfies $A^{h_R}(\Omega_t) \ge \omega(R) - c k\inv$. Then (\ref{eq3}) gives 
\[
[\omega(R)-c k\inv] + \int_{\Omega_t} h_R-h_{R_k} \le \omega(R) + 2 c k\inv + \theta \le \omega(R) + 3ck\inv,
\]
where the last inequality follows provided we assume $\theta \le ck\inv$. 
Rearranging the previous line shows that 
\[
\int_{\Omega_t} h_R - h_{R_k} \le 4c k\inv.
\] 
Hence for $k$ large enough
\[
\int_{\Omega_t} \vert \grad \zeta_R\vert h = \int_{\Omega_t} \frac{d}{dR} h_R \le \frac{5}{4}\int_{\Omega_t} \frac{h_R - h_{R_k}}{k\inv} \le 5c.
\]
Thus (iii) holds provided $k$ is large enough and $\theta \le ck\inv$.  
Moreover, since $A^{h_R}(\Omega_t) \le A^{h_{R_k}}(\Omega_t)$ for all $t$, the inequality (\ref{eq2}) implies that (ii) holds as well.  Finally, it is clear that (i) holds.  Thus $\{\Omega_t\}_{t\in[0,1]}$ belongs to the nice class $\mathcal N_{R,5c,ck\inv,\theta}$ provided $k$ is large enough and $\theta \le ck\inv$.
\end{proof} 

Everything is now set up to carry out the min-max argument. 

\begin{prop}
\label{min-max}
Let $\{R_j\}$ be the sequence from Proposition \ref{radii} and let $h_j = h_{R_j}$ and $\zeta_j = \zeta_{R_j}$. Then for each $j$ there is a smooth, almost embedded $h_{j}$-PMC 
$
\Sigma_j = \bd \Omega_j
$
supported in $B_{R_j+1}$ with $A^{h_j}(\Omega_j) = \omega(R_j)$.  The index of $\Omega_j$ with respect to $A^{h_{j}}$ is at most one and $\int_{\Omega_j} \vert \grad \zeta_j\vert h  \le 10/R_j$.
\end{prop}

\begin{proof}
Fix an integer $j$.  For notational convenience denote $R_j$ by $R$ and $h_{R_j}$ by $h_R$. By Proposition \ref{dbound} and Proposition \ref{nices}, it is possible to find a sequence of mountain pass paths $\{\Omega_t^k\}_{t\in[0,1]}$ which belong to the class $\mathcal N_{R,10/R,\eta,\theta_k}$. Here $\eta > 0$ is a fixed positive number and $\theta_k\to 0$.  Since $\theta_k\to 0$, it follows that 
\[
\sup_{t\in[0,1]} A^{h_R}(\Omega^k_t) \to \omega(R)
\]
as $k\to \infty$.  We would like to apply Zhou's prescribed mean curvature min-max theorem (Theorem 1.7 in \cite{Z}) to the sequence $\{\Omega^k_t\}$ in order to produce an $h_R$-PMC.  Unfortunately, the theorem does not directly apply in our situation.  Nevertheless the proof of the theorem still applies with minor modifications.  We now explain the necessary changes. 

Let $M = \sup_{\R^{n+1}} \vert h\vert$.  Let $\mathcal C$ be the set of all Caccioppoli sets $\Omega \subset B_{R+1-\eps_R}$ with 
\[
\area(\bd \Omega) \le 2\omega(R) + M\vol(B_{R+1-\eps_R}).
\]
Likewise let $\mathcal V$ be the set of all varifolds with support in $\cl B_{R+1-\eps_R}$ and mass bounded by 
$
2\omega(R) + M\vol(B_{R+1-\eps_R}).
$ 
Let $A$ be the set of all $V\in \mathcal V$ with $M$-bounded first variation.  By standard arguments, there exists an $\mathbf F$-continuous tightening map $\psi\f \mathcal C\times [0,1] \to \mathcal C(\R^{n+1})$ with the following properties:
\begin{itemize}
\item[(i)] $\psi(\Omega,0) = \Omega$,
\item[(ii)] $\psi(\Omega,t) = \Omega$ if $\vert \bd \Omega\vert \in A$,
\item[(iii)] $A^{h_R}(\psi(\Omega,t)) \le A^{h_R}(\Omega)$,
\item[(iv)] if $\vert \bd \Omega\vert \notin A$ then 
\[
A^{h_R}(\psi(\Omega,1)) - A^{h_R}(\Omega) \le L(\mathbf F(\vert \bd \Omega\vert ,A))
\]
where $L\f [0,\infty)\to [0,\infty)$ is a continuous function with $L(0) = 0$ and $L(t) > 0$ for $t > 0$,
\item[(v)] $\mathbf F(\psi(\Omega,t),\Omega) \le g(\mathbf F(\vert \bd \Omega\vert,A))$ where $g\f [0,\infty)\to [0,\infty)$ is a continuous function with $g(0) = 0$.
\end{itemize}
For each integer $k$ define the sets 
\begin{gather*}
I_k = \{t\in [0,1]:\, A^{h_R}(\Omega_t^k) \ge \omega(R) - \eta/2\},\\
J_k = \{t\in [0,1]:\, A^{h_R}(\Omega_t^k) > \omega(R) - \eta\}.
\end{gather*}
Then $I_k$ is a compact subset of $J_k$, and $J_k$ is an open subset of $(0,1)$.  Let $\tau_k\f [0,1]\to [0,1]$ be a continuous function which is one on $I_k$ and zero outside of $J_k$.

Let $U^k_t = \psi(\Omega^k_t,\tau_k(t))$. Then $\{U^k_t\}_{t\in[0,1]}$ belongs to $\mathcal P_R$.  By property (iii) of the tightening map, 
\[
\sup_{t\in [0,1]} A^{h_R}(U^k_t) \to \omega(R) 
\]
as $k \to \infty$.  A critical sequence is a sequence $\{U^k_{t_k}\}_{k=1}^\infty$ such that 
\[
A^{h_R}(U^k_{t_k})\to \omega(R), \quad \text{as } k\to\infty.
\]
Define the critical set 
\[
\mathcal K = \big\{V\in \mathcal V:\, V = \lim_{k\to \infty} \vert \bd U^k_{t_k}\vert \text{ for some critical sequence } \{U^k_{t_k}\} \big\}.
\]
We claim that $\mathcal K \subset A$. To see this, suppose that $\{U^k_{t_k}\}$ is a critical sequence and let $V = \lim_{k\to \infty} \vert \bd U^k_{t_k}\vert$.  Notice that 
\[
A^{h_R}(U^k_{t_k}) - A^{h_R}(\Omega^k_{t_k}) \to 0, \quad \text{as } k\to \infty.
\]
By property (iv) of the tightening map and the fact that $t_k\in I_k$ for $k$ large enough, this implies that 
$
\mathbf F(\vert \bd U^k_{t_k}\vert , A) \to 0.
$
The claim follows. 

Again consider a critical sequence $\{U^k_{t_k}\}$.  Since $\mathbf F(\vert \bd U^k_{t_k}\vert , A) \to 0$, property (v) of the tightening map implies that 
$
\vol(U^k_{t_k}\Delta \Omega^k_{t_k}) \to 0
$
as $k\to \infty$. Now observe that 
\begin{align*}
\int_{U^k_{t_k}} \vert \grad \zeta_R\vert h &\le \int_{\Omega^k_{t_k}} \vert \grad \zeta_R\vert  h + \int_{U^k_{t_k} \setminus \Omega^k_{t_k}} \vert \grad \zeta_R \vert h\\
&\le \frac{10}{R} + C\vol(U^k_{t_k}\setminus \Omega^k_{t_k}),\phantom{\int}
\end{align*} 
where we've used the fact that $t_k \in I_k$ and $\{\Omega^k_t\}\in \mathcal N_{R,10/R,\eta,\theta_k}$. 
It follows that 
\begin{equation}
\label{eq6} 
\limsup_{k\to\infty} \int_{U^k_{t_k}} \vert \grad \zeta_R\vert h \le \frac{10}{R},
\end{equation}
and so there is good control over $U^k_{t_k}$ in the transition region.

By hypothesis (H3), the function $h_R$ is an admissible prescription function on $B_{R+1}$. (See \cite{ZZ1} for the definition of admissible prescription functions.)  Moreover, there is no closed, embedded minimal surface contained in the level set $\{h_R = 0\}$.  The almost-minimizing, regularity, and index arguments now carry through unchanged to show that there is a critical sequence $\{U^k_{t_k}\}$ and a Caccioppoli set $\Omega$ such that 
$
\mathbf F(U^k_{t_k},\Omega) \to 0.
$
Moreover, $\Sigma = \bd \Omega$ is a smooth, almost embedded $h_R$-PMC, and $A^{h_R}(\Omega) = \omega(R)$, and the index of $\Omega$ with respect to $A^{h_R}$ is at most one.  Finally $\Omega$ satisfies 
\[
\int_{\Omega} \vert \grad \zeta_j\vert h \le \frac{10}{R}
\]
by equation (\ref{eq6}).  This proves the result. 
\end{proof}

\subsection{Finding the $h$-PMC} 

Again assume the function $h$ satisfies (H1)-(H4).  In this section, we show that, for large $j$, the surfaces $\Sigma_j$ constructed in Proposition \ref{min-max} have a component that does not intersect the transition region $B_{R_j+1}\setminus B_{R_j}$.  The first step is to establish uniform area bounds on the surfaces $\Sigma_j$.  

In the following proof, the idea of differentiating $A^{h_j}$ along the homothetic regions $s\Omega_j$ is taken from \cite{CalMu1}.  This is the only place in the proof where hypothesis (H2) is needed. 

\begin{prop}
There is a constant $C$ which does not depend on $j$ such that $\area(\Sigma_j)\le C$ for all $j$. 
\end{prop}

\begin{proof}
Fix a positive integer $j$.  For each $s > 0$ consider the homothetic  regions $s\Omega_j$. Since $\Sigma_j$ is an $h_{j}$-PMC, it follows that 
\begin{align*}
0 = \frac{d}{ds}\eval_{s=1} A^{h_{j}}(s\Omega_j) & = \frac{d}{ds}\eval_{s=1}\area(s\Sigma_j) - \int_{\Sigma_j} h_{j} X \cdot \nu\\
&= n\area(\Sigma_j) - \int_{\Omega_j} \div(h_{j}X),
\end{align*}
where $\nu$ is the outward normal to $\Sigma_j$ and $X$ denotes the vector field $X(x) = x$.  Now observe that 
\begin{align*}
\div(h_{j}X) &= h_j \div(X) + \grad h_j \cdot X\\
&= (n+1)h_j + \zeta_j \grad h \cdot X + h\grad \zeta_j \cdot X.
\end{align*}
Therefore, 
\begin{align*}
(n+1)A^{h_j}(\Omega_j) &= (n+1)A^{h_j}(\Omega_j) - \frac{d}{ds}\eval_{s=1} A^{h_{j}}(s\Omega_j) \\
&= \area(\Sigma_j) - \int_{\Omega_j} (\zeta_j \grad h \cdot X +  h \grad \zeta_j\cdot X).
\end{align*}
By construction, $\Omega_j$ satisfies 
\[
\int_{\Omega_j} h \grad \zeta_j \cdot X \le 2R_j \int_{\Omega_j} \vert \grad \zeta_j\vert h \le 2R_j \left(\frac{10}{R_j}\right) \le 20.
\]
Combined with the previous equation, this yields 
\[
\area(\Sigma_j) \le (n+1)A^{h_j}(\Omega_j) + 20 + \int_{\Omega_j} \zeta_j \grad h \cdot X.
\]
It follows that 
\begin{equation}
\label{eq4}
(1-\sigma) \area(\Sigma_j) \le \left(n+1-\sigma\right)A^{h_j}(\Omega_j) + 20 + \int_{\Omega_j} (\zeta_j\grad h\cdot X - \sigma  {h_j}).
\end{equation}
But there is an estimate 
\begin{align*}
\int_{\Omega_j} (\zeta_j\grad h\cdot X - \sigma  {h_j}) &= \int_{\Omega_j} \zeta_j(\grad h\cdot X - \sigma  {h})\\
&\le \int_{B_\rho} (\vert \grad h \cdot X\vert + \sigma \vert h\vert) + \int_{\Omega_j \setminus B_\rho} \zeta_j(\grad h \cdot X - \sigma h)
\end{align*}
and assumption (H2) guarantees that $\grad h \cdot X - \sigma h < 0$ outside $B_\rho$. Therefore 
\[
\int_{\Omega_j} (\zeta_j\grad h\cdot X - \sigma  {h_j}) \le C
\]
for a constant $C$ that does not depend on $j$.  Using this estimate in (\ref{eq4}) yields 
\[
(1-\sigma)\area(\Sigma_j) \le (n+1-\sigma)A^{h_j}(\Omega_j) + C.
\]
The desired upper bound for $\area(\Sigma_j)$ follows.
\end{proof}

The following result of Topping \cite{Top} shows that bounds on area and mean curvature imply a bound on diameter. 

\begin{prop}[Special Case of Theorem 1.1 in \cite{Top}]
Assume that $\Sigma^n \subset \R^{n+1}$ is a closed, connected, immersed hypersurface.  Then 
\[
\operatorname{diam}(\Sigma) \le C \int_\Sigma \vert H\vert^{n-1} 
\]
where $C$ depends only on $n$, and $H$ denotes the mean curvature of $\Sigma$. 
\end{prop}

This immediately implies the next corollary. 

\begin{corollary}
\label{diambound}
There is a constant $C$ which does not depend on $j$ such that any connected component $\Gamma_j$ of $\Sigma_j$ satisfies $\operatorname{diam}(\Gamma_j) \le C$.
\end{corollary}

It remains to show that some component of $\Sigma_j$ does not intersect the transition region when $j$ is large. This will complete the proof of Theorem \ref{1}.

\begin{lem}
\label{classi} 
Fix a unit vector $\nu\in \R^{n+1}$ and a constant $\tau$ and let 
\[
h_\infty(x) = n\zeta(x\cdot \nu + \tau).
\]
Every closed $h_\infty$-PMC is a union of spheres of radius 1 contained in the region where $h_\infty = n$. 
\end{lem} 

\begin{proof}
Let $\Sigma = \bd \Omega$ be a compact, almost embedded $h_\infty$-PMC.   Differentiating $A^{h_\infty}$ along the translated regions $\Omega + t\nu$ gives 
\[
0 = \frac{d}{dt}\eval_{t=0} A^{h_\infty}(\Omega + t\nu) = -\frac{d}{dt}\eval_{t=0}\int_{\Omega+t\nu} h_\infty = -\int_\Omega \grad h_\infty \cdot \nu.
\]
Since $h_\infty$ is monotone in the $\nu$ direction, it follows that $\Omega$ is contained in the region where $h_\infty \equiv n$.  Hence, by the well-known result of Alexandrov \cite{Alek}, the $h_\infty$-PMC $\Sigma$ is a union of spheres of radius 1.  
\end{proof}

\begin{prop}
\label{noinf} 
Some connected component of $\Omega_j$ is contained in $B_{R_j}$ when $j$ is large. 
\end{prop}

\begin{proof}
Suppose to the contrary there is a subsequence (not relabeled) such that every component of $\Omega_j$ intersects $B_{R_j + 1}\setminus B_{R_j}$.  Let $\Theta_j$ be a component of $\Omega_j$ and let $p_j$ be a point of $\Gamma_j = \bd\Theta_j$ which is farthest from the origin. Define the translations 
\[
\Gamma_j' = \Gamma_j - p_j, \quad \Theta_j' = \Theta_j - p_j, \quad h_j'(x) = h_j(x+p_j).
\]
Note that the surfaces $\Gamma_j'$ are all contained in a fixed ball $B$ centered at the origin by Corollary \ref{diambound}. 

Passing to a subsequence (not relabeled) and using (H1), there is a unit vector $\nu \in \R^{n+1}$ and a constant $\tau$, and a function $h_\infty(x) = n \zeta(x\cdot \nu + \tau)$ such that $h_j'$ converges smoothly to $h_\infty$ on $B$.  By the compactness theorem for prescribed mean curvature hypersurfaces with bounded area and index (\cite{Z} Theorem 2.6), passing to a subsequence (not relabelled) there is a smooth, almost-embedded  $h_\infty$-PMC $\Gamma_\infty = \bd \Theta_\infty$ contained in $B$ such that $\Gamma_j'$ converges smoothly with multiplicity one to $\Gamma_\infty$ away from at most one point.  In particular, 
$
A^{h_j}(\Theta_j) = A^{h_j'}(\Theta_j') \to A^{h_\infty}(\Theta_\infty).
$

By Lemma \ref{classi}, the set $\Theta_\infty$ is a union of balls of radius 1 contained in the region where $h_\infty = n$. It follows that $A^{h_\infty}(\Theta_\infty) \ge A^n(B_1)$.  Repeating the above argument with any remaining components of $\Omega_j\setminus \Theta_j$ extracts more spheres at infinity.  (There can be at most finitely many such components by the uniform area bound.)  Thus 
\[
\lim_{j\to \infty} A^{h_j}(\Omega_j) \ge A^{n}(B_1),
\]
and this is a contradiction. 
\end{proof}

The following corollary shows that the energy identity holds provided there are no $h$-PMCs with negative energy. 

\begin{corollary}
\label{enid} 
Let $h\f \R^{n+1}\to \R$ be a function satisfying the assumptions of Theorem \ref{main}.  Assume in addition that there are no smooth, almost embedded $h$-PMCs $\Gamma = \bd \Theta$ with $A^h(\Theta) < 0$.  Then there is a smooth, almost embedded $h$-PMC $\Sigma = \bd \Omega$ with $A^h(\Omega) = \omega$.  
\end{corollary}

\begin{proof}
The proof of Proposition \ref{noinf} shows that any component of $\Omega_j$ drifting to infinity carries energy at least $A^n(B_1)$.  Therefore, if all components of $\Omega_j$ have positive energy, it follows that no component of $\Omega_j$ can drift to infinity by (H4). 
\end{proof}

\section{CMCs in Asymptotically Flat Manifolds} 

\label{cmc}

In this section, we show that essentially the same argument can be used to construct CMCs in an asymptotically flat manifold. 

\subsection{Set Up} To begin, we recall the notion of an asymptotically flat 3-manifold. 

\begin{defn}
\label{afd} 
A manifold $(M^3,g)$ is called asymptotically flat if there is a decomposition $M = K \cup E$ where $K$ is compact and $E$ is diffeomorphic to $\R^3\setminus B_R$ for some $R > 0$.  In the corresponding coordinates on $E$, it is further required that 
$
g_{ij} = \delta_{ij} + h_{ij}  
$
where $h_{ij}\to 0$ smoothly as $r\to \infty$ and moreover 
\[
h_{ij} = O(r^{-1}), \quad h_{ij,k} = O(r^{-2})
\]
as $r\to \infty$.  
\end{defn}

For the rest of this section, let $(M^3,g)$ be a fixed asymptotically flat 3-manifold.  Let $M= K \cup E$ be the decomposition of $M$ as per Definition \ref{afd}. Fix a constant $c > 0$.

\begin{defn}
Let $\mathcal P$ be the set of all $c$-mountain pass paths in $M$.  Also let $\omega$ be the one-parameter min-max width for $A^c$. 
\end{defn}

\begin{prop}
The one parameter min-max width satisfies $0 < \omega \le A^c(B_{2/c})$.  Here $B_{2/c}$ is a Euclidean ball of radius $2/c$ and $A^c$ is computed with respect to the Euclidean metric.   
\end{prop}

\begin{proof}
The isoperimetric inequality holds in an asymptotically flat manifold.  Therefore the proof is identical to that of Proposition \ref{estmm}.
\end{proof}

The crucial assumption required for one-parameter min-max to detect a CMC is the following:
\begin{itemize}
\item[(H)] The one parameter width $\omega \text{ is strictly less than } A^c(B_{2/c})$. 
\end{itemize}
We will discuss some geometric conditions which imply that (H) holds in a later section.  For simplicity, in the rest of the paper we will assume that $c = 2$.   In this case, hypothesis (H) becomes that $\omega < A^2(B_1) = \frac{4\pi}{3}$.

\subsection{The Approximate Problems} 

As above fix $c = 2$. Let $(M^3,g)$ be an asymptotically flat manifold satisfying condition (H) for $c=2$, and let $M = K \cup E$ be its decomposition into a compact piece and an asymptotically flat end.  Let $S_r$ be the sphere of Euclidean radius $r$ in $E$, and let $B_r$ be the open, bounded set enclosed by $S_r$ in $M$.  The following asymptotic formula for the mean curvature of $S_r$ was proved by Huang in \cite{Huang}. 

\begin{prop}
The mean curvature of $S_r$ satisfies 
$
H_{S_r} = \frac{2}{r} + O(r^{-2})
$
as $r\to \infty$.
\end{prop}

\begin{proof}
This follows at once from the more precise asymptotic formula derived in Lemma 2.1 of \cite{Huang}.
\end{proof}

Let $N$ be the outward unit normal vector to these spheres $S_r$.  Then $N$ forms an approximate calibration. 

\begin{prop}
The divergence of $N$ satisfies $\div_g(N) = \frac{2}{r} + O(r^{-2})$. In particular, $\div_g(N) \ge \frac{1}{r}$ for all $r$ sufficiently large. 
\end{prop}

\begin{proof}
At a point $p\in S_r$ one computes that 
\begin{align*}
\div(N) &= \div_{T_pS_r}(N) + \la \del_N N,N\ra = H_{S_r}(p) + 0 = \frac{2}{r} + O(r^{-2}),
\end{align*}
as needed.
\end{proof}

The lapse function $\phi$ is the unique function such that the flow of $\phi N$ takes $S_{r_1}$ to $S_{r_2}$ in time $r_2 - r_1$ for all sufficiently large $r_2 \ge r_1$.

\begin{prop}
The lapse function satisfies $\phi(x) \ge \frac{3}{5}$ when $\|x\|$ is sufficiently large. 
\end{prop}

\begin{proof}
This follows from the fact that the lapse function for the Euclidean metric is identically 1, and the fact that $g$ is asymptotically flat. 
\end{proof}

\begin{defn}
Let $\zeta$ be the function from Definition \ref{zetadef}. 
Given $R$ sufficiently large, define a function $\zeta_R$ on $M$ by setting $\zeta_R(x) = \zeta(r(x) - R)$ on $E$ and letting $\zeta_R \equiv 1$ on $K$. 
\end{defn}

The following proposition is easy to verify. 

\begin{prop}
\label{gr}
The gradient of $\zeta_R$ is given by 
\[
\grad \zeta_R(x) = \frac{\zeta'(r(x)-R)}{\phi(x)}N(x).
\] 
In particular, $\vert \grad \zeta_R\vert \le \frac{5}{3}\vert \zeta'(r-R)\vert$ for $r$ sufficiently large. 
\end{prop}

\begin{defn}
Let $\mathcal P_R$ denote the set of $2\zeta_R$-mountain pass paths.  Also let $\omega(R)$ be the one-parameter min-max width for $A^{2\zeta_R}$. 
\end{defn}

The $A^{2\zeta_R}$ functional admits barriers near infinity.  This follows from the fact that $N$ is approximately a calibration.   Let $R\mapsto \eps_R$ be the decreasing function from Definition \ref{epsr} so that 
\[
2\zeta_R(x) \le \frac{1}{r}
\]
whenever $\|x\| \ge R + 1 -2\eps_R$.

\begin{prop}
\label{intersect} 
Let $\Omega\in \mathcal C(M)$ and fix some large $R$.  The intersection $\Omega \cap B_{R+1-2\eps_R}$ belongs to $\mathcal C(M)$ and satisfies $A^{2\zeta_{R}}(\Omega \cap B_{R+1-2\eps_R}) \le A^{2\zeta_{R}}(\Omega)$. 
\end{prop}

\begin{proof} Note that $\div(N) \ge 1/r$ and that $\zeta_R \le 1/r$ outside $B_{R+1-2\eps_R}$.  Therefore the proof is identical to that of Lemma \ref{intersect}, using the vector field $N$ in place of $x/\|x\|$. 
\end{proof}

The next sequence of results copies the ideas from Section \ref{pmc} almost verbatim. 

\begin{defn}
Let $\mathcal Q_R$ denote the set of all $2\zeta_R$-mountain pass paths $\{\Omega_t\}$ such that $\Omega_t \subset B_{R+1-\eps_R}$ for all $t\in [0,1]$. 
\end{defn}

\begin{prop}
The one-parameter min-max width for $A^{2\zeta_R}$ does not change if we restrict to $R$-mountain pass paths contained in $B_{R+1-\eps_R}$.  That is,
\[
\omega(R) = \inf_{\{\Omega_t\}\in \mathcal Q_R} \bigg[\sup_{t\in[0,1]} A^{2\zeta_R}(\Omega_t)\bigg]
\]
\end{prop}

\begin{proof}
Identical to Proposition \ref{support}.
\end{proof} 

Next we employ the monotonicity of $\omega(R)$ to construct mountain pass paths with controlled behavior in the transition region where $\zeta_R$ goes from $1$ to $0$. 

\begin{prop}
The function $R \mapsto \omega(R)$ is decreasing. 
\end{prop}

\begin{proof}
This follows at once from the fact that $\zeta_{R_2} \ge \zeta_{R_1}$ whenever $R_2 > R_1$. 
\end{proof}

\begin{prop}
The approximate min-max numbers satisfy $\omega(R) \to \omega$ as $R \to \infty$. 
\end{prop}

\begin{proof}
Identical to Proposition \ref{radii}.
\end{proof}

\begin{prop}
\label{goodseq}
There is a sequence $R_j\to \infty$ such that $R \mapsto \omega(R)$ is differentiable at each $R_j$ and $0 \ge \omega'(R_j) \ge -2/R_j$. 
\end{prop}

\begin{proof}
Identical to Proposition \ref{dbound}. 
\end{proof} 

\begin{defn}
Given $R > 0$, $C> 0$, $\eta > 0$, and $\theta > 0$ define $\mathcal N_{R,C,\eta,\theta}$ to be the set of $R$-mountain pass paths $\{\Omega_t\}_{t\in[0,1]}$ such that 
\begin{itemize}
\item[(i)] $\Omega_t$ is supported in $B_{R+1-\eps_R}$ for all $t$,
\item[(ii)] $A^{2\zeta_R}(\Omega_t) \le \omega(R) + \theta$ for all $t$, and
\item[(iii)] $\int_{\Omega_t} \vert \grad \zeta_R\vert \le C$ for all $t$ such that $A^{2\zeta_R}(\Omega_t) \ge \omega(R)-\eta$. 
\end{itemize}
\end{defn}

\begin{prop}
\label{niceclass}
Suppose that $0 \ge \omega'(R) \ge -c$ where $c > 0$ is a constant.  Then there is an $\eta > 0$ such that the class $\mathcal N_{R,5c,\eta,\theta}$ is non-empty for all sufficiently small $\theta$. 
\end{prop}

\begin{proof}
Observe that
\[
\frac{d}{dR}\zeta_R = \vert \zeta'(r-R)\vert \le \frac{5}{3} \vert \grad \zeta_R\vert
\]  
when $R$ is large. 
The rest of the proof is identical to Proposition \ref{nices}. 
\end{proof}

Finally we can perform the min-max argument. 

\begin{prop}
\label{minmax}
There is a sequence $R_j \to \infty$ with the following properties.  For each $j$ there is a closed, almost-embedded $2\zeta_{R_j}$-$\operatorname{PMC}$ $\Sigma_j = \bd \Omega_j$ contained in $B_{R+1}$.  This surface satisfies 
\[
A^{2\zeta_{R_j}}(\Omega_j) = \omega(R_j)
\]
and it has index at most one for $A^{2\zeta_{R_j}}$. Moreover, 
\[
\int_{\Omega_j} \vert \grad \zeta_{R_j} \vert \le \frac{10}{R_j}. 
\]
\end{prop}

\begin{proof}
Choose $R_j$ as in Proposition \ref{goodseq}.  Then by Proposition \ref{niceclass}, there is an $\eta > 0$ and a sequence $\theta_k \to 0$ such that the set $\mathcal N_{R_j,10/R_j,\eta,\theta_k}$ is non-empty for all $k$.  For each $k$, let $\{\Omega_t^k\}_{t\in [0,1]}$ be an element in $\mathcal N_{R_j,10/R_j,\eta,\theta_k}$. Applying Zhou and Zhu's PMC min-max theory to this min-max sequence (with the same modifications described in Proposition \ref{min-max}) yields the desired surfaces. 
\end{proof}

\subsection{Finding the CMC}

The goal of this section is to prove that the surfaces $\Sigma_j$ constructed in the previous section are contained in the region $\{\zeta_{R_j} = 1\}$ for all sufficiently large $j$, and therefore give rise to an almost-embedded 2-CMC in $M$. The first step is to prove that the surfaces $\Sigma_j$ satisfy uniform area bounds.  

The next proposition constructs a vector field on $M$ which approximates the coordinate vector field on $\R^3$ near infinity. 

\begin{prop}
\label{vect}
There is a vector field $X$ on $M$ and a compact region $B$ such that 
\begin{itemize}
\item[(i)] $\div_{P} X(x) \le \frac 7 3$ for all points $x\in M$ and all planes $P\in T_xM$, 
\item[(ii)] $\div(X) \ge \frac 8 3$ on $M\setminus B$, and
\item[(iii)] $\vert X\vert \le 2r$ on $M\setminus B$. 
\end{itemize}
\end{prop}

\begin{proof}
Let $Y = x^i \frac{\bd}{\bd x^i}$ be the Euclidean coordinate vector field.  With respect to the Euclidean metric, every Euclidean unit vector $v$ satisfies 
$
g_{\text{euc}}(\del^{\text{euc}}_v Y, v) = 1. 
$
Since $g$ is asymptotically flat, it follows that 
\begin{equation}
\label{e9}
\vert g(\del_w Y, w) - 1\vert \le \frac{1}{100}
\end{equation}
whenever $w\in T_xM$ satisfies $g(w,w) = 1$ and $\|x\|$ is sufficiently large.  Moreover, since $\vert Y\vert_{g_{\text{euc}}} = r$, it follows that 
\begin{equation}
\label{e7}
\vert Y\vert_g \le 2r 
\end{equation}
on the complement of a compact region.  Let $A$ be a compact region so that (\ref{e9}) and (\ref{e7}) hold on $M\setminus A$. 

Select a radially increasing cutoff function $\eta$ which equals zero on $A$, and equals one on the complement of a compact set $B\supset A$.  Assuming $B$ is chosen sufficiently large, it is also possible to ensure that 
\begin{equation}
\label{e8} 
\vert \grad \eta\vert \le \frac{1}{r \log r}
\end{equation}
everywhere on $M$. 
Define $X = \eta Y$.  It is easy to see that $X$ satisfies (iii).  Moreover, for any vector $w$ one has 
\[
g(\del_w X,w) = g(\grad \eta, w) g(Y,w) + \eta g(\del_w Y,w). 
%\label{e10}
\]
Combining this  with (\ref{e9}), (\ref{e7}), and (\ref{e8}), it follows that $X$ satisfies (i) and (ii). 
\end{proof} 

Computing the first variation of $A^{2\zeta_j}(\Sigma_j)$ along $X$ lets one bound the area of $\Sigma_j$ by $A^{2\zeta_j}(\Sigma_j)$.  This gives the uniform area bound. 

\begin{prop}
Let $\Sigma_j = \bd \Omega_j$ be the surfaces constructed in Proposition \ref{minmax}. There is a constant $C$ independent of $j$ such that $\area(\Sigma_j) \le C$ for all $j$.  
\end{prop}

\begin{proof}
For notational convenience let $\zeta_j = \zeta_{R_j}$. Let $X$ be the vector field from Proposition \ref{vect}. Then 
\begin{align*}
0 = \delta A^{2\zeta_j}\eval_{\Omega_j}(X) &= \int_{\Sigma_j} \div_{T_p\Sigma_j}(X) - 2\int_{\Sigma_j} \zeta_j X \cdot \nu \\
&= \int_{\Sigma_j} \div_{T_p\Sigma_j}(X) - 2\int_{\Omega_j} \div(\zeta_j X).
\end{align*}
By the product rule
\[
\div(\zeta_j X) = \la \grad \zeta_j, X\ra + \zeta_j \div(X).
\]
Therefore 
\begin{align*}
\frac 8 3 A^{2\zeta_j}(\Omega_j) &= \frac 8 3 \area(\Sigma_j) - \frac{16} 3 \int_{\Omega_j}\zeta_j\\
&\qquad\quad - \int_{\Sigma_j} \div_{T_p\Sigma_j}(X) + 2\int_{\Omega_j} \la \grad \zeta_j,X\ra + 2\int_{\Omega_j}\zeta_j \div(X)\\
&\ge \frac 1 3 \area(\Sigma_j) - {4}R_j\int_{\Omega_j} \vert \grad \zeta_j\vert + 2 \int_{\Omega_j} \zeta_j \left(\div X - \frac 8 3\right)\\
&\ge \frac 1 3 \area(\Sigma_j) - 4R_j\int_{\Omega_j} \vert \grad \zeta_j\vert - C\vol(B).
\end{align*}
Since both $A^{2\zeta_j}(\Omega_j)$ and $4 R_j\int_{\Omega_j}\vert \grad \zeta_j\vert$ are uniformly bounded, the proposition follows. 
\end{proof}

The area and index bounds allow one to study the pointed convergence of the surfaces $\Sigma_j$ as $j\to \infty$.

\begin{prop}
The diameter of every connected component of $\Sigma_j$ is uniformly bounded. 
\end{prop} 

\begin{proof}
Fix a very large number $R > 0$.  Since $g$ is asymptotically flat, the Euclidean area of $\Sigma_j \cap (M\setminus B_R)$ is uniformly bounded.  Likewise, the Euclidean mean curvature of $\Sigma_j$ is also uniformly bounded on $M\setminus B_R$.  Since the boundary of $\Sigma_j \cap (M\setminus B_R)$ is contained in $S_R$, a result of Simon implies that the diameter of any component of $\Sigma_j$ is uniformly bounded.  Namely, given a surface $\Sigma$ immersed in $\R^3$ with bounded mean curvature, the proof of Lemma 1.1 in \cite{SimonDiam} gives constants $r_0 > 0$ and $C > 0$ such that 
$
\area(\Sigma \cap B_{r_0}(x)) \ge C
$
whenever $x\in \Sigma$ and $\bd \Sigma\cap B_{r_0}(x) = \emptyset$.  Here $B_{r_0}(x)$ denotes the Euclidean ball of radius $r_0$ centered at $x$, and the constant $C$ depends only on the upper bound for the mean curvature.  Applying this with $\Sigma = \Sigma_j \cap (M\setminus B_R)$ proves the proposition. 
\end{proof}

The following proposition completes the proof of the Theorem \ref{main2}. 

\begin{prop}
Some component of $\Sigma_j$ is contained in $B_{R_j}$ when $j$ is sufficiently large. 
\end{prop} 

\begin{proof}
Pick a sequence of points $p_j\in M$ diverging to infinity.  Let $B$ be a Euclidean ball of a fixed large radius.  Define the translated functions $\zeta_j'(x) = \zeta_j(p_j+x)$ and the translated metrics $g_{ij}'(x) = g_{ij}(p_j+x)$ on $B$.  Then $g_{ij}'$ converges smoothly to the flat metric.  Also there is a unit vector $\nu$ and a constant $\tau$ such that, up to a subsequence, $\zeta_j'$ converges smoothly to $\zeta(x\cdot \nu + \tau)$.  Therefore the proof is identical to Proposition \ref{noinf}.
\end{proof}

As in the prescribed mean curvature case, the energy identity holds provided there are no 2-CMCs with negative energy. 

\begin{corollary}
Let $(M^3,g)$ be an asymptotically flat manifold satisfying the assumptions of Theorem \ref{main2}.  Assume that there are no smooth, almost embedded $c$-CMCs $\Gamma = \bd \Theta$ with $A^c(\Theta) < 0$. Then there is a smooth, almost embedded $c$-CMC $\Sigma = \bd \Omega$ with $A^c(\Omega) = \omega$. 
\end{corollary}

\begin{proof}
Identical to Corollary \ref{enid}.
\end{proof}

\subsection{Examples} 

In this section, we construct a family of examples to which Theorem \ref{main2} applies.  Let $M = \R^3$ and fix a function $v \f \R^3\to \R$ such that 
$
v = {1}/{r}
$
near infinity.  Then the conformal metrics 
\[
g_t = (1+tv)^4 g_{\text{euc}}
\]
are asymptotically flat. Let $B_r$ denote a ball of radius $r$ centered at the origin in the Euclidean metric.  For $t > 0$, the sign of the scalar curvature of $g_t$ at $x$ is equal to the sign of $-\lap v(x)$. 

\begin{prop}
Assume that $\lap v < 0$ on $B_2$.  Then 
\[
\frac{d}{dt}\eval_{t=0} \bigg[\area_{g_t}(\bd B_1) - 2 \vol_{g_t}(B_1)\bigg] < 0.
\]
\end{prop}

\begin{proof}
Observe that 
\[
\area_{g_t}(\bd B_1) - 2 \vol_{g_t}(B_1) = \int_{\bd B_1} (1+tv)^4 - 2 \int_{B_1} (1+tv)^6,
\]
where the integrals are taken with respect to the Euclidean surface measure and the Euclidean volume respectively. Thus 
\[
\frac{d}{dt}\eval_{t=0} \bigg[\area_{g_t}(\bd B_1) - 2 \vol_{g_t}(B_1)\bigg] = 4\int_{\bd B_1} v - 12\int_{B_1} v.
\]
Define the auxiliary function 
\[
\phi(r) = \frac{1}{r^2}\int_{\bd B_r} v - \frac{3}{r^3}\int_{B_r} v.
\]
To prove the proposition, we must check that $\phi(1) < 0$. To this end, note that 
\begin{align*}
\phi'(r) &= \frac{1}{r^2}\int_{\bd B_r} \bd_\nu v - \frac{3}{r^3} \int_{\bd B_r}v + \frac{9}{r^4} \int_{B_r} v\\
&= \frac{1}{r^2}\int_{B_r} \lap v - \frac{3}{r}\phi(r). 
\end{align*} 
Multiplying this by $r^3$ and using the fact that $\lap v < 0$ gives that
\[
r^3 \phi'(r) + 3r^2 \phi(r) < 0.
\]
Finally, integrating and using the fact that $\phi(r) \to 0$ as $r\to 0$ proves that $r^3\phi(r) < 0$ for all $r \in (0,2]$. In particular, this implies that $\phi(1) < 0$, as needed. 
\end{proof}

This has the following corollary. 

\begin{corollary}
Assume that $\lap v < 0$ on $B_2$.  Then there is an $\eps = \eps(v) > 0$ such that the one parameter min-max width for $A^2$ satisfies
\[
\omega(g_t) < \frac{4\pi}{3} 
\]
for all $t\in (0,\eps)$. In particular, $(M,g_t)$ contains a closed, almost-embedded $2$-CMC for all $t\in(0,\eps)$ by Theorem \ref{main2}. 
\end{corollary}

\end{document}